\documentclass[reqno]{amsart}

\usepackage{amsmath}
\usepackage{amsfonts}
\usepackage{amssymb,enumerate}
\usepackage{amsthm}
\usepackage[all]{xy}
\usepackage{rotating}
\usepackage{hyperref}
\usepackage{color}
\theoremstyle{plain}
\newtheorem{lem}{Lemma}[section]
\newtheorem{cor}[lem]{Corollary}
\newtheorem{prop}[lem]{Proposition}
\newtheorem{thm}[lem]{Theorem}

\theoremstyle{definition}
\newtheorem{defn}[lem]{Definition}
\newtheorem{ex}[lem]{Example}

\newtheorem{disc}[lem]{Remark}

\newtheorem{notn}[lem]{Notation}
\newtheorem{fact}[lem]{Fact}

\newcommand{\cat}[1]{\mathcal{#1}}

\newcommand{\catd}{\cat{D}}

\newcommand{\pd}{\operatorname{pd}}

\newcommand{\id}{\operatorname{id}}	
\newcommand{\fd}{\operatorname{fd}}

\newcommand{\depth}{\operatorname{depth}}	
	
\newcommand{\amp}{\operatorname{amp}}

\newcommand{\mspec}{\operatorname{m-Spec}}

\newcommand{\HH}{\operatorname{H}}
\newcommand{\Hom}{\operatorname{Hom}}	

\newcommand{\spec}{\operatorname{Spec}}

\newcommand{\shift}{\mathsf{\Sigma}}

\newcommand{\cone}{\operatorname{Cone}}

\newcommand{\ideal}[1]{\mathfrak{#1}}
\newcommand{\m}{\ideal{m}}

\newcommand{\p}{\ideal{p}}

\newcommand{\fm}{\ideal{m}}

\newcommand{\fa}{\ideal{a}}
\newcommand{\fb}{\ideal{b}}

\newcommand{\ol}{\overline}

\newcommand{\supp}{\operatorname{supp}}
\newcommand{\Supp}{\operatorname{Supp}}

\newcommand{\VE}{\operatorname{V}}
\newcommand{\cosupp}{\operatorname{co-supp}}

\newcommand{\bbz}{\mathbb{Z}}

\newcommand{\xra}{\xrightarrow}

\newcommand{\vf}{\varphi}

\newcommand{\y}{\mathbf{y}}

\newcommand{\x}{\underline{x}}

\renewcommand{\geq}{\geqslant}
\renewcommand{\leq}{\leqslant}

\newcommand{\Rhom}[3][R]{\mathbf{R}\!\operatorname{Hom}_{#1}(#2,#3)}	
\newcommand{\Lotimes}[3][R]{#2\otimes^{\mathbf{L}}_{#1}#3}
\newcommand{\Otimes}[3][R]{#2\otimes_{#1}#3}
\renewcommand{\Hom}[3][R]{\operatorname{Hom}_{#1}(#2,#3)}	
\newcommand{\Tor}[4][R]{\operatorname{Tor}^{#1}_{#2}(#3,#4)}

\newcommand{\LL}[2]{\mathbf{L}\Lambda^{\ideal{#1}}(#2)}

\newcommand{\RG}[2]{\mathbf{R}\Gamma_{\ideal{#1}}(#2)}

\newcommand{\width}{\operatorname{width}}

\newcommand{\Comp}[2]{\widehat{#1}^{\ideal{#2}}}
\newcommand{\rad}[1]{\operatorname{rad}(#1)}

\newcommand{\catdfb}{\catd_{\text{b}}^{\text{f}}}
\newcommand{\catdb}{\catd_{\text{b}}}
\newcommand{\catdf}{\catd^{\text{f}}}

\newcommand{\LLno}[1]{\mathbf{L}\Lambda^{\ideal{#1}}}

\newcommand{\RGno}[1]{\mathbf{R}\Gamma_{\ideal{#1}}}

\newcommand{\fromRG}[2]{\varepsilon_{\ideal{#1}}^{#2}}
\newcommand{\fromRGno}[1]{\varepsilon_{\ideal{#1}}}
\newcommand{\toLL}[2]{\vartheta^{\ideal{#1}}_{#2}}
\newcommand{\toLLno}[1]{\vartheta^{\ideal{#1}}}

\newcommand{\PRhom}[3][R]{\mathbf{R}\!\operatorname{Hom}_{#1}\left(#2,#3\right)}	
\newcommand{\PRG}[2]{\mathbf{R}\Gamma_{\ideal{#1}}\left(#2\right)}
\newcommand{\PLL}[2]{\mathbf{L}\Lambda^{\ideal{#1}}\left(#2\right)}

\numberwithin{equation}{lem}

\begin{document}

\bibliographystyle{amsplain}

\author{Sean Sather-Wagstaff}

\address{Department of Mathematical Sciences,
Clemson University,
O-110 Martin Hall, Box 340975, Clemson, S.C. 29634
USA}

\email{ssather@clemson.edu}

\urladdr{https://ssather.people.clemson.edu/}

\thanks{
Sean Sather-Wagstaff was supported in part by a grant from the NSA}

\author{Richard Wicklein}

\address{Richard Wicklein, Mathematics and Physics Department, MacMurray College, 447 East College Ave., Jacksonville, IL 62650, USA}

\email{richard.wicklein@mac.edu}

\title{Adically Finite Chain Complexes}

%\date{\today}

%\dedicatory{}


\begin{abstract}
We investigate the similarities between adic finiteness and homological finiteness for chain complexes over a commutative noetherian ring. 
In particular, we extend the isomorphism properties of certain natural morphisms from homologically finite complexes
to adically finite complexes. 
We do the same for characterizations of certain homological dimensions.
In addition, we study transfer of adic finiteness along ring homomorphisms, 
all with a view toward subsequent applications.
%We investigate the similarities between adic finiteness and homological finiteness for chain complexes over a commutative noetherian ring. In particular, we extend the isomorphism properties of certain natural morphisms from homologically finite complexes to adically finite complexes. We do the same for characterizations of certain homological dimensions. In addition, we study transfer of adic finiteness along ring homomorphisms, all with a view toward subsequent applications.
% notes: 20 pages. part 2 of a series with http://arxiv.org/abs/1401.6925 and http://arxiv.org/abs/1506.07052. comments welcome
% add other arxiv refs after posting
\end{abstract}

\maketitle

\tableofcontents

\section{Introduction} \label{sec130805a}
Throughout this paper let $R$ and $S$ be 
commutative noetherian rings, let $\fa \subsetneq R$ be a proper ideal of $R$, and let $\Comp{R}{a}$ be the $\fa$-adic completion of $R$.
Let $\x=x_1,\ldots,x_n\in R$ be a generating sequence for $\fa$, and consider the Koszul complex $K:=K^R(\x)$.
We work in the derived category $\catd(R)$ with objects the $R$-complexes
indexed homologically
$X=\cdots\to X_i\to X_{i-1}\to\cdots$.
See, e.g., \cite{hartshorne:rad,verdier:cd,verdier:1} and Section~\ref{sec140109b} for background/foundational material.
Isomorphisms in $\catd(R)$ are identified by the symbol $\simeq$.

\

This work is part 2 of a series of papers about support and finiteness conditions for complexes,
with a view toward derived local cohomology and homology.
It builds on our previous paper~\cite{sather:scc}, and it is used in the 
papers~\cite{sather:afbha, sather:afc,sather:asc,sather:elclh}.

In~\cite{sather:scc} we introduce the notion of ``$\fa$-adic finiteness'' for complexes; see Definition~\ref{def120925d}.
For example, an $R$-module $M$ is $\fa$-adically finite if it is $\fa$-torsion and has $\Tor i{R/\fa}M$ finitely generated
for all $i$. In particular, this recovers two standard finiteness conditions as special cases:
first, $M$ is finitely generated (i.e., noetherian) if and only if it is $0$-adically finite;
second, over a local ring, $M$ is artinian if and only if it is adically finite with respect to the ring's maximal ideal. 

One goal of this paper is to extend standard results for finitely generated modules (and homologically finite complexes)
to the $\fa$-adically finite realm. For instance, given a finitely generated $R$-module $M$,
a classical result shows that $\Tor iM-$ commutes not only with arbitrary direct sums, but also
with arbitrary direct products.
One of our main results extends this to the case where $M$ is $\fa$-adically finite.

\begin{thm}\label{thm151128a}
Let $M$ be an $\fa$-adically finite $R$-complex, and
let  $V\in\catd(R)$ such that $\supp_R(V)\subseteq\VE(\fa)$. 
Consider a set $\{N_\lambda\}_{\lambda\in\Lambda}\in\catd_+(R)$ 
such that $\HH_i(N_\lambda)=0$ for all $i< s$ and for all $\lambda\in\Lambda$. 
Consider the natural morphism
$$\Lotimes M{\prod_\lambda N_\lambda}\xra p\prod_\lambda(\Lotimes M{N_\lambda})$$
in $\catd(R)$.
Then the induced morphism
$$\Lotimes V{\Lotimes M{\prod_\lambda N_\lambda}}\to\Lotimes V{\prod_\lambda(\Lotimes M{N_\lambda})}$$
is an isomorphism.
\end{thm}

This is contained in Theorem~\ref{lem151126b} from the body of the paper. 
Note the trade-off in this result, as compared to the classical one. We have relaxed the assumptions on $M$,
but we are not claiming that the morphism $p$ is an isomorphism, only that a certain induced morphism is so.
This is the theme of the results of Section~\ref{sec151104b}.
While these results may seem quite specialized, we exhibit applications in~\cite{sather:afc}.

Theorem~\ref{thm151128a} uses Foxby's ``small support'', as do most of the results of this paper;
see Definition~\ref{defn130503a}. 
This is an extremely useful substitute for the standard notion of support (our ``large support'') for finitely generated modules.
For instance, large support allows us to give conditions on two finitely generated modules to decide when their tensor product is non-zero.
This does not work in general for non-finitely generated modules, but one can use small support to detect when at least one of their Tor-modules
is non-zero, which ends up being enough for many applications. 
Thus, small support provides another substitute for finite generation. 

The paper continues with Section~\ref{sec151008a} which tracks some transfer behavior of these notions, specifically,
support, cosupport, and adic finiteness through restriction and extension of scalars.
These are used heavily in the paper~\cite{sather:asc}. 

Section~\ref{sec151206a} contains other results showing how similar adic finiteness is to homological finiteness with respect to homological dimensions. 
For instance, the next result, contained in Theorem~\ref{prop151115a} below is well known when $X$ is homologically finite; it is somewhat surprising to us that it holds in this generality.

\begin{thm}\label{prop151115aa}
Let $X\in\catdb(R)$ be $\fa$-adically finite. If $X$ is locally of finite flat dimension, then $\pd_R(X)<\infty$.
Moreover, one has $\pd_R(X)=\fd_R(X)$.
\end{thm}

\section{Background}\label{sec140109b}

\subsection*{Derived Categories}
The $i$th shift (or suspension) of an $R$-complex $X$ is denoted $\shift^iX$.
We  consider the following full triangulated subcategories of $\catd(R)$.

\

$\catd_+(R)$: objects are the complexes $X$ with $\HH_i(X)=0$ for $i\ll 0$.

$\catd_-(R)$: objects are the complexes $X$ with $\HH_i(X)=0$ for $i\gg 0$. 

$\catdb(R)$: objects are the complexes $X$ with $\HH_i(X)=0$ for $|i|\gg 0$.

$\catdf(R)$: objects are the complexes $X$ with $\HH_i(X)$ finitely generated for all $i$.

\

\noindent Doubly ornamented subcategories are intersections, e.g., $\catdfb(R):=\catdf(R)\bigcap\catdb(R)$.

\subsection*{Resolutions}
An $R$-complex $P$ 
is \emph{semi-projective}\footnote{In the literature, semi-projective complexes are sometimes called ``K-projective'' or ``DG-projective''.} 
if it respects surjective quasiisomorphisms,
that is, if it consists of projective $R$-modules and respects quasiisomorphisms; see~\cite[1.2.P]{avramov:hdouc}.
A \emph{semi-projective resolution} of an $R$-complex $X$ is a quasiisomorphism $P\xra\simeq X$ such that $P$ is semi-projective.
The \emph{projective dimension} of $X$ is finite, written $\pd_R(X)<\infty$, if it has a bounded semi-projective resolution.
The corresponding flat and injective versions of these notions (homotopically flat, etc.) are defined similarly. 

For the following items, consult~\cite[Section 1]{avramov:hdouc} or~\cite[Chapters 3 and 5]{avramov:dgha}.
Bounded below  complexes of projective modules are semi-projective, 
bounded below  complexes of flat modules are semi-flat, and
bounded above  complexes of injective modules are semi-injective, 
Semi-projective $R$-complexes are semi-flat.
Every $R$-complex admits a semi-projective resolution (hence, a semi-flat one) and a semi-injective resolution.

\subsection*{Derived Functors}
The right derived functor of Hom is $\Rhom --$, which is computed via a semi-projective resolution in the first slot
or a semi-injective resolution in the second slot. 
The left derived functor of tensor product is $\Lotimes --$, which is computed via semi-flat resolutions in either slot.

Let $\Lambda^{\fa}$ denote the $\fa$-adic completion functor, and
$\Gamma_{\fa}$ is the $\fa$-torsion functor, i.e.,
for an $R$-module $M$ we have
$$\Lambda^{\fa}(M)=\Comp Ma
\qquad
\qquad
\qquad
\Gamma_{\fa}(M)=\{ x \in M \mid \fa^{n}x=0 \text{ for } n \gg 0\}.$$ 
A module $M$ is \textit{$\fa$-torsion} if $\Gamma_{\fa}(M)=M$.

The associated left and right derived functors (i.e., \emph{derived local homology and cohomology} functors)
are  $\LL a-$ and $\RG a-$.
Specifically, given an $R$-complex $X\in\catd(R)$ and a semi-flat resolution $F\xra\simeq X$ and a 
semi-injective resolution $X\xra\simeq I$, then we have $\LL aX\simeq\Lambda^{\fa}(F)$ and $\RG aX\simeq\Gamma_{\fa}(I)$.
Note that these definitions yield natural transformations $\RGno a\xra{\fromRGno a}\id\xra{\toLLno a} \LLno a$, induced by the natural morphisms
$\Gamma_{\fa}(I)\xra{\iota_{\fa}^I} I$ and $F\xra{\nu^{\fa}_F} \Lambda^{\fa}(F)$.
These notions go back to Grothendieck~\cite{hartshorne:lc}, and Matlis~\cite{matlis:kcd,matlis:hps}, respectively;
see also~\cite{lipman:lhcs,lipman:llcd}.

\begin{fact}\label{fact130619b}
By~\cite[Theorem~(0.3) and Corollary~(3.2.5.i)]{lipman:lhcs}, there are natural isomorphisms
of functors
\begin{align*}
\RG a-\simeq\Lotimes{\RG aR}{-}&&
\LL a-\simeq\Rhom{\RG aR}{-}.
\end{align*}
\end{fact}

\subsection*{Koszul Complexes}
We refer to the following as the ``self-dual nature'' of $K$.

\begin{disc}\label{disc151211a}
Recall the isomorphism
$$K\cong\shift^n\Hom KR.$$
Given an $R$-complex $X$, there are natural isomorphisms
$$\Otimes KX\cong\shift^n\Otimes{\Hom KR}X\cong\shift^n\Hom KX$$
in the category of $R$-complexes;
these are verified by induction on $n$, using the definitions $K(x_i)\cong\cone(R\xra{x_i}R)$
and $K:= K^R(x_1)\otimes_R\cdots\otimes_R K(x_n)$, in terms of mapping cones and tensor products.
From these, we have isomorphisms in $\catd(R)$
\begin{gather*}
K\simeq\shift^n\Rhom KR\\
\Lotimes KX\simeq\shift^n\Lotimes{\Rhom KR}X\simeq\shift^n\Rhom KX.
\end{gather*}
Similarly, one verifies the next natural evaluation isomorphisms for all $X,Y\in\catd(R)$:
\begin{gather*}
\Rhom{\Rhom KX}Y
\simeq\Lotimes K{\Rhom XY}
\simeq\Rhom X{\Lotimes KY}.
\end{gather*}
\end{disc}

\subsection*{Support and Co-support}
The following notions are due to Foxby~\cite{foxby:bcfm} and Benson, Iyengar, and Krause~\cite{benson:csc}.

\begin{defn}\label{defn130503a}
Let $X\in\catd(R)$.
The \emph{small and large support} and  \emph{small co-support} of $X$ are
\begin{align*}
\operatorname{supp}_R(X)
&=\{\mathfrak{p} \in \operatorname{Spec}(R)\mid \Lotimes{\kappa(\p)}X\not\simeq 0 \} \\
\operatorname{Supp}_R(X)
&=\{\mathfrak{p} \in \operatorname{Spec}(R)\mid \Lotimes{R_{\p}}X\not\simeq 0 \} \\
\cosupp_{R}(X)
&=\{\mathfrak{p} \in \operatorname{Spec}(R)\mid \Rhom{\kappa(\p)}X\not\simeq 0 \} 
\end{align*}
where $\kappa(\p):=R_\p/\p R_\p$.
We have a notion of $\operatorname{Co-supp}_R(X)$, as well, but do not need it in the current paper.
\end{defn}

Much of the following is from~\cite{foxby:bcfm} when $X$ and $Y$ are appropriately bounded 
and from~\cite{benson:lcstc,benson:csc} in general. We refer to~\cite{sather:scc} as a matter of convenience.

\begin{fact}\label{cor130528a}
Let $X,Y\in\catd(R)$. Then we have $\supp_R(X)=\emptyset$ if and only if $X\simeq 0$ if and only if $\cosupp_R(X)=\emptyset$,
because of~\cite[Fact~3.4 and Proposition~4.7(a)]{sather:scc}.
Also, by~\cite[Propositions~3.12 and~4.10]{sather:scc} we have 
\begin{align*}
\supp_{R}(\Lotimes{X}{Y}) 
&= \supp_R(X)\bigcap\supp_R(Y)\\
\cosupp_{R}(\Rhom{X}{Y}) 
&= \supp_R(X)\bigcap\cosupp_R(Y).
\end{align*}
In addition, we know that $\supp_R(X)\subseteq\VE(\fa)$ if and only if 
the natural morphism $\fromRG aX\colon\RG aX\to X$ is an isomorphism,
that is, if and only if each homology module $\HH_i(X)$ is $\fa$-torsion,
by~\cite[Proposition~5.4]{sather:scc}
and~\cite[Corollary~4.32]{yekutieli:hct}.
Similarly, we have
$\cosupp_R(X)\subseteq\VE(\fa)$ if and only if 
the natural morphism $\toLL aX\colon X\to \LL aX$ is an isomorphism,
by~\cite[Propositions 5.9]{sather:scc}.
\end{fact}

\subsection*{Adic Finiteness}
The next fact and definition from~\cite{sather:scc} take their cues from work of 
Hartshorne~\cite{hartshorne:adc},
Kawasaki~\cite{kawasaki:ccma,kawasaki:ccc}, and
Melkersson~\cite{melkersson:mci}.

\begin{fact}[\protect{\cite[Theorem 1.3]{sather:scc}}]
\label{thm130612a}
For $X\in\catd_{\text b}(R)$, the next conditions are equivalent.
\begin{enumerate}[\rm(i)]
\item\label{cor130612a1}
One has $\Lotimes{K^R(\underline{y})}{X}\in\catdfb(R)$  for some (equivalently, for every) generating sequence $\underline{y}$ of $\fa$.
\item\label{cor130612a2}
One has  $\Lotimes{X}{R/\mathfrak{a}}\in\catd^{\text{f}}(R)$.
\item\label{cor130612a3}
One has  $\Rhom{R/\mathfrak{a}}{X}\in\catd^{\text{f}}(R)$.
\end{enumerate}
\end{fact}

\begin{defn}\label{def120925d}
An $R$-complex $X\in\catdb(R)$ is \emph{$\mathfrak{a}$-adically finite} if it satisfies the equivalent conditions of Fact~\ref{thm130612a} and $\operatorname{supp}_R(X) \subseteq \operatorname{V}(\mathfrak{a})$.
\end{defn}

\begin{ex}\label{ex160206a}
Let $X\in\catdb(R)$ be given.
\begin{enumerate}[(a)]
\item \label{ex160206a1}
If $X\in\catdfb(R)$, then we have $\supp_R(X)=\VE(\fb)$ for some ideal $\fb$, and it follows that $X$ is $\fa$-adically finite
whenever $\fa\subseteq\fb$. (The case $\fa=0$ is from~\cite[Proposition~7.8(a)]{sather:scc}, and the general case follows readily.)
\item \label{ex160206a2}
$K$ and $\RG aR$ are $\fa$-adically finite, by~\cite[Fact~3.4 and Theorem~7.10]{sather:scc}.
\item \label{ex160206a3}
If $(R,\m)$ is local, then the homology modules of $X$ are artinian if and only if $X$ is $\m$-adically finite, by~\cite[Proposition~7.8(b)]{sather:scc}.
See Proposition~\ref{prop160208a} for an extension of this.
\end{enumerate}
\end{ex}

\subsection*{Bookkeeping}
We use some convenient accounting tools due to Foxby~\cite{foxby:ibcahtm}.

\begin{defn}\label{defn151112a}
The
\emph{supremum}, \emph{infimum},  \emph{amplitude}, \emph{$\fa$-depth}, and \emph{$\fa$-width} of an $R$-complex $Z$ are
\begin{align*}
\sup(Z)&=\sup\{ i\in\bbz\mid\HH_i(Z)\neq 0\}\\
\inf(Z)&=\inf\{ i\in\bbz\mid\HH_i(Z)\neq 0\}\\
\amp(Z)&=\sup(Z)-\inf(Z)\\
\depth_{\fa}(Z)&=-\sup(\Rhom{R/\fa}Z)\\
\width_{\fa}(Z)&=\inf(\Lotimes{(R/\fa)}Z)
\end{align*}
with the conventions $\sup\emptyset=-\infty$ and $\inf\emptyset=\infty$.
\end{defn}

\begin{fact}\label{disc151112a}
Let $Y,Z\in\catd(R)$.
\begin{enumerate}[(a)]
\item\label{disc151112a1}
By definition, one has $\sup(Z)<\infty$ if and only if $Z\in\catd_-(R)$.
Also, one has $\inf(Z)>-\infty$ if and only if $Z\in\catd_+(R)$, and 
one has $\amp(Z)<\infty$ if and only if $Z\in\catdb(R)$.
\item\label{disc151112a5}
By~\cite[Lemma~2.1]{foxby:ibcahtm}, there are inequalities
\begin{gather*}
\inf(Y)+\inf(Z)\leq\inf(\Lotimes YZ) \\
\sup(\Rhom YZ)\leq\sup(Z)-\inf(Y).
\end{gather*}
\end{enumerate}
\end{fact}

\section{Koszul Homology}\label{sec151206b}

We begin this section by showing how, with appropriate support conditions, bounded Koszul homology implies bounded homology.
Note that the self-dual nature~\ref{disc151211a} of the Koszul complex 
implies that these also give results for Koszul cohomology; see, e.g., the proof of Lemma~\ref{lem150604a2}.

\begin{lem}\label{lem150604a1}
Let  $Z\in\catd(R)$, let $\y=y_1,\ldots,y_m\in\fa$, and set
$L:=K^R(\y)$ and $\fb=(\y)R$.
Assume that $\supp_R(Z)\subseteq\VE(\fa)$, e.g., that each homology module $\HH_i(Z)$ is annihilated by a power of $\fa$.
\begin{enumerate}[\rm(a)]
\item\label{lem150604a1z}
There are (in)equalities
\begin{align*}
\inf(\Lotimes LZ)&\leq m+\inf(Z)&
\sup(\Lotimes LZ)&=m+\sup(Z)\\
\amp(Z)&\leq\amp(\Lotimes LZ)&
\depth_{\fb}(Z)&=-\sup(Z).
\end{align*}
\item\label{lem150604a1a}
For each $*\in\{+,-,\text{b}\}$, one has $\Lotimes LZ\in\catd_*(R)$ if and only if $Z\in\catd_*(R)$.
\end{enumerate}
\end{lem}

\begin{proof}
Claim: For any $i\in\bbz$, if $\HH_i(\Lotimes{K^R(y_1)}{Z})=0$, then $\HH_{i-1}(Z)=0$.
To prove the claim, assume that $\HH_i(\Lotimes{K^R(y_1)}{Z})=0$.
The standard long exact sequence for Koszul homology contains the following:
$$\HH_i(\Lotimes{K^R(y_1)}{Z})\to\HH_{i-1}(Z)\xra{y_1}\HH_{i-1}(Z).$$
It follows that the map $\HH_{i-1}(Z)\xra{y_1}\HH_{i-1}(Z)$ is injective. 
However, the condition $\supp_R(Z)\subseteq\VE(\fa)$ implies that $\HH_{i-1}(Z)$ is $\fa$-torsion
by Fact~\ref{cor130528a}. Since $y_1$ is in $\fa$, the fact that the map $\HH_{i-1}(Z)\xra{y_1}\HH_{i-1}(Z)$ is injective
therefore implies that $\HH_{i-1}(Z)=0$, as claimed.

\eqref{lem150604a1z}
We now show that $\inf(\Lotimes LZ)\leq m+\inf(Z)$.
For this, it suffices to show that for any $i\in\bbz$, if $\HH_i(\Lotimes{L}{Z})=0$, then $\HH_{i-m}(Z)=0$.
We verify this by induction on $m$, the base case $m=1$ being covered by the above claim.
For the induction step, it suffices to note that Fact~\ref{cor130528a}
implies that $\supp_R(\Lotimes{K^R(y_1)}Z)\subseteq\supp_R(Z)\subseteq\VE(\fa)$.

For the equality $\sup(\Lotimes LZ)=m+\sup(Z)$,
note that the condition $\supp_R(Z)\subseteq\VE(\fa)\subseteq\VE(\fb)$ implies that $\RG bZ\simeq Z$,
by Fact~\ref{cor130528a}.
Thus, we have the following equalities by~\cite[Theorem 2.1]{foxby:daafuc}
$$\sup(\Lotimes LZ)=\sup(\RG bZ)+m=\sup(Z)+m.$$

The inequality for amp follows directly.
For the equality $\depth_{\fb}(Z)=\sup(Z)$, note that~\cite[Theorem 2.1]{foxby:daafuc}
shows that we have
$$\depth_{\fb}(Z)=-\sup(\Lotimes LZ)+m
=-\sup(Z)$$
by what we have already shown.

\eqref{lem150604a1a} This follows from part~\eqref{lem150604a1z}, 
via Fact~\ref{disc151112a}.
\end{proof}

Note that some items in the next result use $L$, while others use $K$. 

\begin{lem}\label{lem150604a2}
Let  $Z\in\catd(R)$, let $\y=y_1,\ldots,y_m\in\fa$, and set
$L:=K^R(\y)$ and $\fb:=(\y)R\subseteq\fa$. 
Assume that $\cosupp_R(Z)\subseteq\VE(\fa)$, e.g., that each homology module $\HH_i(Z)$ is annihilated by a power of $\fa$.
\begin{enumerate}[\rm(a)]
\item\label{lem150604a2z}
There are (in)equalities
\begin{gather*}
\width_\fb(Z)=\inf(Z)=\inf(\Lotimes LZ)\\
\sup(Z)-n\leq\sup(\Lotimes KZ) \\
\amp(Z)-n\leq\amp(\Lotimes KZ).
\end{gather*}
\item\label{lem150604a2b}
For each $*\in\{+,-,\text{b}\}$, one has $\Lotimes KZ\in\catd_*(R)$ if and only if $Z\in\catd_*(R)$.
\end{enumerate}
\end{lem}

\begin{proof}
\eqref{lem150604a2z}
The assumption $\cosupp_R(Z)\subseteq\VE(\fa)\subseteq\VE(\fb)$ implies that $Z\simeq\LL bZ$, by Fact~\ref{cor130528a}. 
This explains the first equality in the next sequence
$$\inf(Z)=\inf(\LL bZ)=\inf(\Lotimes LZ)=\inf(\Lotimes{(R/\fb)}Z)=\width_\fb(Z)$$
while the remaining equalities are 
from~\cite[Theorem 4.1]{foxby:daafuc} and by definition.
This also explains the first isomorphism in the next sequence
$$Z\simeq\LL aY\simeq\LL a{\RG aZ}\simeq\Rhom{\RG aR}{\RG aZ}$$
while the other isomorphisms are from~\cite[Theorem~(0.3)$^*$ and Corollary]{lipman:lhcs}.
This explains the first step in the next sequence.
\begin{align*}
\sup(Z)
&=\sup(\Rhom{\RG aR}{\RG aZ})\\
&\leq\sup(\RG aZ)-\inf(\RG aR) \\
&\leq\sup(\Lotimes KZ)+n
\end{align*}
The second step is from Fact~\ref{disc151112a}\eqref{disc151112a5}.
The third step follows from the equality $\sup(\RG aZ)=\sup(\Lotimes KY)$ of~\cite[Theorem~2.1]{foxby:daafuc},
and the inequality $\inf(\RG aR)\geq-n$ which is via the \v Cech complex. 
This explains the first two rows of (in)equalities from part~\eqref{lem150604a2z},
and the third row follows by definition.

\eqref{lem150604a2b}
This follows from part~\eqref{lem150604a2z} and Fact~\ref{disc151112a}.
\end{proof}

\begin{disc}\label{disc151114a}
Since $Z\simeq 0$ if and only if
$\inf(Z)=\infty$, the previous two results also have the following conclusions:
one has $Z\simeq 0$ if and only if $\Lotimes KZ\simeq 0$. However, we already know this 
because of Remark~\ref{disc151211a} and Fact~\ref{cor130528a}.
\end{disc}

We continue with some useful computations of homological dimensions.
Note that the next result shows that the quantities $\fd_R(\Lotimes LZ)$ and
$\fd_R(Z)$ are simultaneously finite, and similarly for $\pd$.

\begin{lem}\label{lem151206a}
Let  $Z\in\catd(R)$, let $\y=y_1,\ldots,y_m\in\fa$, and set
$L:=K^R(\y)$.
Assume that $\supp_R(Z)\subseteq\VE(\fa)$.
Then we have 
\begin{gather*}
\fd_R(\Lotimes LZ)=m+\fd_R(Z)\\
\pd_R(\Lotimes LZ)=m+\pd_R(Z).
\end{gather*}
\end{lem}

\begin{proof}
Let $N$ be an $R$-module, and note that $\supp_R(\Lotimes NZ)\subseteq\supp_R(Z)\subseteq\VE(\fa)$.
Using this with the associativity isomorphism
$\Lotimes N{(\Lotimes LZ)}\simeq\Lotimes L{(\Lotimes NZ)}$
we conclude from  Lemma~\ref{lem150604a1}\eqref{lem150604a1z} that
\begin{align*}
\sup(\Lotimes N{(\Lotimes LZ)})
&=\sup(\Lotimes L{(\Lotimes NZ)})
=m+\sup(\Lotimes NZ).
\end{align*}
Thus, it follows from~\cite[Proposition~2.4.F]{avramov:hdouc}
that we have
\begin{align*}
\fd_R(\Lotimes LZ)
&=\sup\{\sup(\Lotimes N{(\Lotimes LZ)})\mid\text{$N$ is an $R$-module}\}\\
&=m+\sup\{\sup(\Lotimes NZ)\mid\text{$N$ is an $R$-module}\}\\
&=m+\fd_R(Z).
\end{align*}
One verifies the equality $\pd_R(\Lotimes LZ)=m+\pd_R(Z)$ similarly.
\end{proof}

The next result is verified like the previous one.

\begin{lem}\label{lem151206b}
Let  $Z\in\catd(R)$, let $\y=y_1,\ldots,y_m\in\fa$, and set
$L:=K^R(\y)$.
Assume that $\cosupp_R(Z)\subseteq\VE(\fa)$.
Then we have 
$\id_R(\Rhom LZ)=m+\id_R(Z)$.
\end{lem}

\section{Induced Isomorphisms}\label{sec151104b}

This section contains Theorem~\ref{thm151128a} from the introduction, with several other results of the same ilk. 
The main idea is to replace the homologically finite assumption
with $\fa$-adically finiteness in some well-known isomorphism theorems.
We begin with tensor-evaluation.

\begin{thm}\label{lem150525a}
Let $M$ be an $\fa$-adically finite $R$-complex,
and let $Y\in\catd_-(R)$ and $U,V,Z\in\catd(R)$.
Assume that $\supp_R(V),\cosupp_R(U)\subseteq\VE(\fa)$ and that either $\fd_R(M)$ or $\fd_R(Z)$ is finite.
Consider an exact triangle 
$$\Lotimes{\Rhom MY}{Z}\xra{\omega_{MYZ}}\Rhom M{\Lotimes YZ}\to C\to $$
in $\catd(R)$ where
$\omega_{MYZ}$ is the natural tensor evaluation morphism.
\begin{enumerate}[\rm(a)]
\item \label{lem150525a1}
One has $\supp_R(C)\bigcap\VE(\fa)=\emptyset=\cosupp_R(C)\bigcap\VE(\fa)$.
\item \label{lem150525a2}
The  morphisms $\Lotimes V{\omega_{MYZ}}$ and $\Rhom V{\omega_{MYZ}}$ and $\Rhom {\omega_{MYZ}}U$ are isomorphisms.
\end{enumerate}
\end{thm}

\begin{proof}
The fact that $M$ is $\fa$-adically finite implies that $\Rhom KM\in\catdfb(R)$.
If $\fd_R(M)<\infty$, then $\fd_R(\Rhom KM)<\infty$, that is, $\pd_R(\Rhom KM)<\infty$ because of the homological finiteness.
Thus, when either $\fd_R(M)$ or $\fd_R(Z)$ is finite,
the tensor evaluation morphism $\omega_{\Rhom KMYZ}$ in the next commutative diagram
$$\xymatrix@C=15mm{
\Lotimes K{\Lotimes{\Rhom MY}{Z}}
\ar[r]^-{\Lotimes{\theta_{KMY}}Z}_-\simeq
\ar[d]_{\Lotimes K{\omega_{MYZ}}}
&\Lotimes{\Rhom{\Rhom KM}Y}Z 
\ar[d]^{\omega_{\Rhom KMYZ}}_-\simeq \\
\Lotimes{K}{\Rhom M{\Lotimes YZ}}
\ar[r]^-{\theta_{KM\Lotimes YZ}}_-\simeq
&\Rhom{\Rhom KM}{\Lotimes YZ}.
}$$
is an isomorphism in $\catd(R)$, by~\cite[Lemma~4.4(F)]{avramov:hdouc};
this result also explains the horizontal isomorphisms.
It follows that $\Lotimes K{\omega_{MYZ}}$ is an isomorphism as well.
We conclude that $\Lotimes KC\simeq 0$, so by Fact~\ref{cor130528a} we have
$$\emptyset=\supp_R(\Lotimes KC)=\supp_R(K)\bigcap\supp_R(C)=\VE(\fa)\bigcap\supp_R(C).$$
This  implies that $\emptyset=\supp_R(V)\bigcap\supp_R(C)$,
so $\Lotimes VC\simeq 0$.
The induced triangle 
$$\Lotimes V{\Lotimes{\Rhom MY}{Z}}\xra{\Lotimes V{\omega_{MYZ}}}\Lotimes V{\Rhom M{\Lotimes YZ}}\to \Lotimes VC\to $$
shows that $\Lotimes V{\omega_{MYZ}}$ is an isomorphism.
The isomorphism $\Rhom {\omega_{MYZ}}U$ is verified similarly.

On the other hand, the  self-dual nature~\ref{disc151211a} of $K$ implies that $\Rhom K-\simeq\shift^{-n}\Lotimes K-$,
so the induced morphism $\Rhom K{\omega_{MYZ}}$ is also an isomorphism.
We conclude as above that
$\emptyset=\VE(\fa)\bigcap\cosupp_R(C)$ \footnote{See also~\cite[Corollary~4.9]{benson:csc}.}
and that 
the morphism $\Rhom{V}{\omega_{MYZ}}$ is  an isomorphism, as desired.
\end{proof}

Next, we consider Hom-evaluation.

\begin{thm}\label{lem150525b}
Let $M$ be an $\fa$-adically finite $R$-complex,
and let 
$Y\in\catdb(R)$ and $U,V,Z\in\catd(R)$.
Assume that $\supp_R(V),\cosupp_R(U)\subseteq\VE(\fa)$ and that either $\fd_R(M)$ or $\id_R(Z)$ is finite.
Consider an exact triangle
$$\Lotimes M{\Rhom YZ}\xra{\theta_{MYZ}}\Rhom {\Rhom MY}Z\to C\to$$
in $\catd(R)$ where $\theta_{MYZ}$ is the natural Hom-evaluation morphism.
\begin{enumerate}[\rm(a)]
\item \label{lem150525b1}
One has $\supp_R(C)\bigcap\VE(\fa)=\emptyset=\cosupp_R(C)\bigcap\VE(\fa)$.
\item \label{lem150525b2}
The  morphisms $\Lotimes V{\theta_{MYZ}}$ and $\Rhom V{\theta_{MYZ}}$ and $\Rhom {\theta_{MYZ}}U$ are isomorphisms.
\end{enumerate}
\end{thm}

\begin{proof}
The fact that $M$ is $\fa$-adically finite implies that $\Lotimes KM\in\catdfb(R)$.
If $\fd_R(M)<\infty$, then $\fd_R(\Lotimes KM)<\infty$, that is, $\pd_R(\Lotimes KM)<\infty$ because of the homological finiteness.
Thus, when either $\fd_R(M)$ or $\id_R(Z)$ is finite,
the morphism $\theta_{\Lotimes KMYZ}$ is an isomorphism in $\catd(R)$, by~\cite[Lemma~4.4(I)]{avramov:hdouc}.
This result also explains the second vertical isomorphism in the following commutative diagram 
$$\xymatrix@C=6mm{
\Lotimes{\Lotimes KM}{\Rhom YZ}
\ar[r]^-{\Lotimes K{\theta_{MYZ}}}\ar[d]_{\theta_{\Lotimes KMYZ}}^\simeq
&\Lotimes K{\Rhom{\Rhom MY}Z} 
\ar[d]^{\theta_{K\Rhom MY Z}}_\simeq \\
\Rhom{\Rhom{\Lotimes KM}Y}Z 
\ar[r]^-\simeq
&\Rhom{\Rhom{K}{\Rhom MY}}Z
}$$
in $\catd(R)$.
The unspecified isomorphism is induced by Hom-tensor adjointness.
It follows that $\Lotimes K{\theta_{MYZ}}$ is an isomorphism, and
the rest of the desired conclusions follow as in the proof of Theorem~\ref{lem150525a}.
\end{proof}

Similarly, we have the next result, via~\cite[Proposition~2.2]{christensen:apac}.

\begin{thm}\label{lem151115a}
Let $M$ be an $\fa$-adically finite $R$-complex,
and let $U,V,Y\in\catd(R)$ and $Z\in\catd_-(R)$ be such that $\supp_R(V),\cosupp_R(U)\subseteq\VE(\fa)$.
Assume that at least one of the following conditions holds:
\begin{enumerate}[\rm(1)]
\item\label{lem151115a3}
$\pd_R(Y)<\infty$, or
\item\label{lem151115a4}
$Z\in\catdb(R)$ and $\fd_R(M)<\infty$.
\end{enumerate}
Consider an exact triangle 
$$\Lotimes{\Rhom YZ}{M}\xra{\omega_{YZM}}\Rhom Y{\Lotimes ZM}\to C\to $$
in $\catd(R)$ where
$\omega_{MYZ}$ is the natural tensor evaluation morphism.
\begin{enumerate}[\rm(a)]
\item \label{lem151115a1}
One has $\supp_R(C)\bigcap\VE(\fa)=\emptyset=\cosupp_R(C)\bigcap\VE(\fa)$.
\item \label{lem151115a2}
The  morphisms $\Lotimes V{\omega_{YZM}}$ and $\Rhom V{\omega_{YZM}}$ and $\Rhom {\omega_{YZM}}U$ are isomorphisms.
\end{enumerate}
\end{thm}

Next, we document some special cases of the previous results. 

\begin{cor}\label{lem151114a}
Let $M$ be an $\fa$-adically finite $R$-complex, and
let $Y\in\catd_-(R)$ and $U,V\in\catd(R)$ such that $\supp_R(V),\cosupp_R(U)\subseteq\VE(\fa)$. Let $R\to S$ be a ring homomorphism.
Assume 
that either $\fd_R(M)<\infty$ or  $\fd_R(S)<\infty$.
Consider an exact triangle
$$\Lotimes{\Rhom MY}{S}\xra{\alpha_{MYS}}\Rhom[S] {\Lotimes MS}{\Lotimes YS}\to C\to$$
in $\catd(R)$ where $\alpha_{MYS}$ is the natural morphism.
\begin{enumerate}[\rm(a)]
\item \label{lem151114a1}
One has $\supp_R(C)\bigcap\VE(\fa)=\emptyset=\cosupp_R(C)\bigcap\VE(\fa)$.
\item \label{lem151114a2}
The  morphisms $\Lotimes V{\alpha_{MYZ}}$ and $\Rhom V{\alpha_{MYZ}}$ and $\Rhom {\alpha_{MYZ}}U$ are isomorphisms.
\end{enumerate}
\end{cor}

\begin{proof}
Consider the following commutative diagram in $\catd(R)$
$$\xymatrix{
\Lotimes{\Rhom MY}{S}
\ar[r]^-{\omega_{MYS}}\ar[d]_{\alpha_{MYS}}
&\Rhom M{\Lotimes YS} \\
\Rhom[S] {\Lotimes MS}{\Lotimes YS}\ar[r]^-\simeq
&
\Rhom M{\Rhom[S]{S}{\Lotimes YS}}. \ar[u]_\simeq
}$$
The unspecified isomorphisms are  Hom-cancellation and Hom-tensor adjointness.
Apply the functors $\Lotimes V-$ and $\Rhom V-$ and $\Rhom -U$ to this diagram, and use Theorem~\ref{lem150525a}
to prove part~\eqref{lem151114a2}. Then part~\eqref{lem151114a1} follows from Fact~\ref{cor130528a}.
\end{proof}

The next two lemmas are probably well-known. 
In the absence of suitable references, we include some proof.

\begin{lem}\label{lem151126a}
Let $X\in\catdf_+(R)$, and consider a set $\{N_\lambda\}_{\lambda\in\Lambda}\in\catd_+(R)$ 
such that $\inf(N_\lambda)\geq s$ for all $\lambda\in\Lambda$. 
Then the natural morphism
$$\Lotimes X{\prod_\lambda N_\lambda}\to\prod_\lambda(\Lotimes X{N_\lambda})$$
is an isomorphism in $\catd(R)$.
\end{lem}

\begin{proof}
Let $F\xra\simeq X$ be a degree-wise finite semi-free resolution, that is, a quasiisomorphism where
$F$ is a bounded below complex of finitely generated free $R$-modules.
Truncate each $N_\lambda$ if necessary to assume that $(N_\lambda)_q=0$ for all $q< s$.
Since each $F_p$ is a finite-rank free module, the natural map
$$\Otimes {F_p}{\prod_\lambda}(N_\lambda)_q\to\prod_\lambda(\Otimes {F_p}{(N_\lambda})_q)$$
is an isomorphism for each $q$.
Thus, the product map
$$\prod_{p+q=i}\left(\Otimes {F_p}{\prod_\lambda(N_\lambda)_q}\right)
\to
\prod_{p+q=i}\left(\prod_{\lambda}(\Otimes {F_p}{(N_\lambda)_q})\right)
$$
is an isomorphism as well, for each $i$.
Using the natural isomorphism
$$\prod_{p+q=i}\left(\prod_{\lambda}\Otimes {F_p}{(N_\lambda)_q}\right)
\cong
\prod_\lambda\left(\prod_{p+q=i}\Otimes {F_p}{(N_\lambda)_q}\right)
$$
we conclude that the natural map
\begin{equation}\label{eq151126a}
\prod_{p+q=i}\left(\Otimes {F_p}{\prod_\lambda(N_\lambda)_q}\right)
\to
\prod_\lambda\left(\prod_{p+q=i}\Otimes {F_p}{(N_\lambda)_q}\right)
\end{equation}
is an isomorphism a well.

Our boundedness assumptions on $F$ and $N_\lambda$ imply that, for each $i$, we have
\begin{gather*}
\bigoplus_{p+q=i}\left(\Otimes {F_p}{\prod_\lambda(N_\lambda)_q}\right)
=
\prod_{p+q=i}\left(\Otimes {F_p}{\prod_\lambda(N_\lambda)_q}\right)\\
\prod_\lambda\left(\bigoplus_{p+q=i}\Otimes {F_p}{(N_\lambda)_q}\right)
=
\prod_\lambda\left(\prod_{p+q=i}\Otimes {F_p}{(N_\lambda)_q}\right)
\end{gather*}
so the isomorphism~\eqref{eq151126a}  with these equalities shows that the natural map
$$
\bigoplus_{p+q=i}\left(\Otimes {F_p}{\prod_\lambda (N_\lambda)_q}\right)
\to
\prod_\lambda\left(\bigoplus_{p+q=i}\Otimes {F_p}{(N_\lambda)_q}\right)
$$
is an isomorphism for each $i$. 
Thus, the chain map
$$
\Otimes {F}{\prod_\lambda N_\lambda}
\to
\prod_\lambda\left(\Otimes {F}{N_\lambda}\right)
$$
is an isomorphism. 
By design, this represents
the natural morphism
$$\Lotimes X{\prod_\lambda}N_\lambda\to\prod_\lambda(\Lotimes X{N_\lambda})$$
in $\catd(R)$, so this morphism is an isomorphism, as desired.
\end{proof}

The next result is proved like the previous one.

\begin{lem}\label{lem151126c}
Let $X\in\catdf_+(R)$, and consider a set $\{N_\lambda\}_{\lambda\in\Lambda}\in\catd_-(R)$ 
such that $\sup(N_\lambda)\leq s$ for all $\lambda\in\Lambda$. 
Then the natural morphism
$$
\bigoplus_\lambda(\PRhom X{N_\lambda})
\to
\PRhom X{\bigoplus_\lambda N_\lambda}
$$
is an isomorphism in $\catd(R)$.
\end{lem}

Next, we soup up the previous two results, first by proving Theorem~\ref{thm151128a} from the introduction.

\begin{thm}\label{lem151126b}
Let $M$ be an $\fa$-adically finite $R$-complex, and
let  $U,V\in\catd(R)$ be such that $\supp_R(V),\cosupp_R(U)\subseteq\VE(\fa)$. 
Consider a set $\{N_\lambda\}_{\lambda\in\Lambda}\in\catd_+(R)$ 
such that $\inf(N_\lambda)\geq s$ for all $\lambda\in\Lambda$. 
Consider an exact triangle 
$$\Lotimes M{\prod_\lambda N_\lambda}\xra p\prod_\lambda(\Lotimes M{N_\lambda})\to A\to$$
where $p$ is the natural morphism.
\begin{enumerate}[\rm(a)]
\item \label{lem151126b1}
One has $\supp_R(A)\bigcap\VE(\fa)=\emptyset=\cosupp_R(A)\bigcap\VE(\fa)$.
\item \label{lem151126b2}
The  morphisms $\Lotimes V{p}$ and $\Rhom V{p}$ and $\Rhom {p}U$ are isomorphisms.
\end{enumerate}
\end{thm}

\begin{proof}
As in our previous results, it suffices to show that $\Lotimes Kp$ is an isomorphism in $\catd(R)$.
As $M$ is $\fa$-adically finite, we have $K,\Lotimes KM\in\catdfb(R)$
and $\inf(\Lotimes M{N_{\lambda}})\geq s+\inf(M)$ by Fact~\ref{disc151112a}\eqref{disc151112a5}.
Thus, Lemma~\ref{lem151126a} explains the unspecified isomorphisms in the next commutative diagram.
$$\xymatrix{
\Lotimes {\Lotimes KM}{\prod_\lambda N_\lambda}\ar[rd]^\simeq\ar[d]_{\Lotimes Kp}
\\
\Lotimes K{\prod_\lambda(\Lotimes {M}{N_\lambda})}
\ar[r]^-\simeq
&\prod_\lambda(\Lotimes {\Lotimes KM}{N_\lambda})
}$$
We conclude that $\Lotimes Kp$ is also an isomorphism, as desired.
\end{proof}

The next result follows similarly from Lemma~\ref{lem151126c}.

\begin{thm}\label{lem151126d}
Let $M$ be an $\fa$-adically finite $R$-complex, and
let  $U,V\in\catd(R)$ such that $\supp_R(V),\cosupp_R(U)\subseteq\VE(\fa)$. 
Consider a set $\{N_\lambda\}_{\lambda\in\Lambda}\in\catd_-(R)$ 
such that $\sup(N_\lambda)\leq s$ for all $\lambda\in\Lambda$. 
Consider an exact triangle 
$$
\bigoplus_\lambda(\PRhom M{N_\lambda})
\xra q
\PRhom M{\bigoplus_\lambda N_\lambda}\to B\to
$$
where $q$ is the natural morphism. 
\begin{enumerate}[\rm(a)]
\item \label{lem151126d1}
One has $\supp_R(B)\bigcap\VE(\fa)=\emptyset=\cosupp_R(B)\bigcap\VE(\fa)$.
\item \label{lem151126d2}
The  morphisms $\Lotimes V{p}$ and $\Rhom V{p}$ and $\Rhom {p}U$ are isomorphisms.
\end{enumerate}
\end{thm}

The case $M=\RG aR$ is useful for subsequent results, so we document it next. 

\begin{cor}\label{lem150612a}
Let $U,V,Z\in\catd(R)$ such that $\cosupp_R(U),\supp_R(V)\subseteq\VE(\fa)$.
\begin{enumerate}[\rm(a)]
\item\label{lem150612a1}
If $Y\in\catd_-(R)$, then there are natural isomorphisms in $\catd(R)$
\begin{gather*}
\Lotimes V{\Lotimes{\LL aY}Z}\xra\simeq\Lotimes V{\LL a{\Lotimes YZ}}\\
\Rhom V{\Lotimes{\LL aY}Z}\xra\simeq\Rhom V{\LL a{\Lotimes YZ}}\\
\Rhom {\LL a{\Lotimes YZ}}U
\xra\simeq
\Rhom {\Lotimes{\LL aY}Z}U.
\end{gather*}
\item\label{lem150612a2}
If $Y\in\catdb(R)$, then there are natural isomorphisms in $\catd(R)$
\begin{gather*}
\Lotimes V{\RG a{\Rhom YZ}}\xra\simeq\Lotimes V{\Rhom {\LL aY}{Z}}\\
\Rhom V{\RG a{\Rhom YZ}}\xra\simeq\Rhom V{\Rhom {\LL aY}{Z}}\\
\Rhom {\Rhom {\LL aY}{Z}}U\xra\simeq\Rhom {\RG a{\Rhom YZ}}U.
\end{gather*}
\item\label{lem150612a4}
If $Z\in\catdb(R)$ or both $Z\in\catd_-(R)$ and $\pd_Y(R)<\infty$, then there are natural isomorphisms in $\catd(R)$
\begin{gather*}
\Lotimes V{\RG a{\Rhom YZ}}\xra\simeq\Lotimes V{\Rhom {Y}{\RG aZ}}\\
\Rhom V{\RG a{\Rhom YZ}}\xra\simeq\Rhom V{\Rhom {Y}{\RG aZ}}\\
\Rhom {\Rhom {Y}{\RG aZ}}U\xra\simeq\Rhom {\RG a{\Rhom YZ}}U.
\end{gather*}
\item\label{lem150612a3}
Let $R\to S$ be a ring homomorphism.
If $Y\in\catd_-(R)$, then there are natural isomorphisms in $\catd(S)$
\begin{gather*}
\Lotimes V{\Lotimes{\LL aY}S}\xra\simeq\Lotimes V{\mathbf{L}\Lambda^{\fa S}(\Lotimes YS)}\\
\Rhom V{\Lotimes{\LL aY}S}\xra\simeq\Rhom V{\mathbf{L}\Lambda^{\fa S}(\Lotimes YS)}\\
\Rhom {\mathbf{L}\Lambda^{\fa S}(\Lotimes YS)}U\xra\simeq\Rhom {\Lotimes{\LL aY}S}U.
\end{gather*}
\item\label{lem150612a5}
Given a set $\{N_\lambda\}_{\lambda\in\Lambda}\in\catd_+(R)$ 
such that $\inf(N_\lambda)\geq s$ for all $\lambda\in\Lambda$,
there are natural isomorphisms in $\catd(R)$
\begin{gather*}
\Lotimes V{\PRG a{\prod_\lambda N_\lambda}}\xra\simeq\Lotimes V{\prod_\lambda\RG a{N_\lambda}}\\
\PRhom V{\PRG a{\prod_\lambda N_\lambda}}\xra\simeq\PRhom V{\prod_\lambda\RG a{N_\lambda}}\\
\PRhom {\prod_\lambda\RG a{N_\lambda}}U\xra\simeq\PRhom {\PRG a{\prod_\lambda N_\lambda}}U.
\end{gather*}
\item\label{lem150612a6}
Given a set $\{N_\lambda\}_{\lambda\in\Lambda}\in\catd_+(R)$ 
such that $\inf(N_\lambda)\geq s$ for all $\lambda\in\Lambda$,
there are natural isomorphisms in $\catd(R)$
\begin{gather*}
\Lotimes V{\bigoplus_\lambda\LL a{N_\lambda}}\xra\simeq\Lotimes V{\PLL a{\bigoplus_\lambda N_\lambda}}\\
\PRhom V{\bigoplus_\lambda\LL a{N_\lambda}}\xra\simeq\PRhom V{\PLL a{\bigoplus_\lambda N_\lambda}}\\
\PRhom {\PLL a{\bigoplus_\lambda N_\lambda}}U\xra\simeq\PRhom {\bigoplus_\lambda\LL a{N_\lambda}}U.
\end{gather*}
\end{enumerate}
\end{cor}

\begin{proof}
The complex $M=\RG aR$ is $\fa$-adically finite by Example~\ref{ex160206a}\eqref{ex160206a2}
and has finite flat dimension via the \v Cech complex.
Thus, the result follows from Fact~\ref{fact130619b} and the preceding results.
\end{proof}

\section{Transfer of Support, Co-support, and Finiteness} \label{sec151008a}

In this section, we consider various transfer relations along a ring homomorphism.

\begin{notn}\label{notn151011a}
Throughout this section, let $\vf\colon R\to S$ be a ring homomorphism such that $\fa S\neq S$,
and consider the forgetful functor $Q\colon\catd(S)\to\catd(R)$.
\end{notn}

\subsection*{Restriction of Scalars}

\newcommand{\LLS}[2]{\mathbf{L}\Lambda^{\ideal{#1} S}(#2)}
\newcommand{\RGS}[2]{\mathbf{R}\Gamma_{\ideal{#1} S}(#2)}

We begin this subsection with two useful isomorphisms.

\begin{lem}\label{lem151127a}
Given an $S$-complex $Y\in\catd(S)$, there are natural isomorphisms
\begin{align*}
Q(\LLS aY)\simeq \LL a{Q(Y)} &&
Q(\RGS aY)\simeq \RG a{Q(Y)}.
\end{align*}
\end{lem}

\begin{proof}
In the following computation, the first and last isomorphisms are from Fact~\ref{fact130619b}, and the
second one comes from \v Cech complexes.
\begin{align*}
Q(\LLS aY)
&\simeq Q(\Rhom[S]{\RGS aS}Y) \\
&\simeq Q(\Rhom[S]{\Lotimes S{\RG aR}}Y) \\
&\simeq Q(\Rhom{\RG aR}{\Rhom[S]SY}) \\
&\simeq \Rhom{\RG aR}{Q(Y)} \\
&\simeq \LL a{Q(Y)}
\end{align*}
The other isomorphisms are from adjointness and cancellation. The explains the first isomorphism, and the second one
is verified similarly.
\end{proof}

\begin{lem}\label{lem151008a}
Given an $S$-complex $Y\in\catd(S)$, one has $\supp_S(Y)\subseteq\VE(\fa S)$ if and only if $\supp_R(Q(Y))\subseteq\VE(\fa)$.
\end{lem}

\begin{proof}
From Lemma~\ref{lem151127a}, the natural morphism
$\RG a{Q(Y)}\to Q(Y)$ is equivalent to the induced morphism $Q(\RGS aY)\to Q(Y)$.
In particular, this says that the morphism
$\RG a{Q(Y)}\to Q(Y)$ is an isomorphism in $\catd(R)$ if and only if $Q(\RGS aY)\to Q(Y)$ is an isomorphism in $\catd(R)$,
that is, if and only if the morphism $\RGS aY\to Y$ is an isomorphism in $\catd(S)$.
Two applications of Fact~\ref{cor130528a} then give the desired result. 

Alternately, we know that $\supp_S(Y)\subseteq\VE(\fa S)$ if and only if each homology module $\HH_i(Y)$ is $\fa S$-torsion, and similarly for $Q(Y)$;
see Fact~\ref{cor130528a}.
Since we have $\HH_i(Q(Y))\cong\HH_i(Y)$, these are $\fa$-torsion if and only if they are $\fa S$-torsion, so the result follows.
\end{proof}

\begin{lem}\label{lem151008b}
Given an $S$-complex $Y\in\catd(S)$, 
one has $\cosupp_S(Y)\subseteq\VE(\fa S)$ if and only if $\cosupp_R(Q(Y))\subseteq\VE(\fa)$.
\end{lem}

\begin{proof}
Argue as in the first paragraph of the proof of the previous result.
\end{proof}

\begin{thm}\label{lem151009a}
Let $Y\in\catd(S)$ be given. 
If $Q(Y)$ is $\fa$-adically finite over $R$, then $Y$ is $\fa S$-adically finite over $S$;
the converse holds when the induced map $\ol\vf\colon R/\fa\to S/\fa S$ is module finite, e.g., when $\vf$ is module finite or $S=\Comp Ra$.
\end{thm}

\begin{proof}
It is clear that we have $Q(Y)\in\catdb(R)$ if and only if $Y\in\catdb(S)$,
and we have $\supp_S(Y)\subseteq\VE(\fa S)$ if and only if $\supp_R(Q(Y))\subseteq\VE(\fa)$ by Lemma~\ref{lem151008a}.
Note that $K':=\Lotimes SK$ is the Koszul complex over $S$ on a finite generating sequence for $\fa S$.
Consider the following isomorphisms:
$$
\Lotimes[S]{K'}{Y}
\simeq\Lotimes[S]{(\Lotimes SK)}{Y}
\simeq\Lotimes K{Q(Y)}.
$$
If each homology module $\HH_i(\Lotimes K{Q(Y)})$ is finitely generated over $R$, then the above isomorphisms imply that $\HH_i(\Lotimes[S]{K'}{Y})$
is finitely generated over $R$, hence over $S$. Thus, if $Q(Y)$ is $\fa$-adically finite over $R$, then $Y$ is $\fa S$-adically finite over $S$, by definition.

For the converse, assume that $Y$ is $\fa S$-adically finite over $S$ and that the induced map $\ol\vf$ is module finite.
Since each module $\HH_i(\Lotimes[S]{K'}{Y})$ is finitely generated over $S$ and has a natural $S/\fa S$-module structure, it is finitely generated over $S/\fa S$.
Our finiteness assumption on $S/\fa S$ conspires with the above isomorphisms to imply that each module $\HH_i(\Lotimes K{Q(Y)})$ is finitely generated over $R/\fa$,
hence over $R$, so 
$Q(Y)$ is $\fa$-adically finite over $R$.
\end{proof}

The next result 
explains the relation between the condition $\vf^*(\supp_S(F))\supseteq\VE(\fa)\bigcap\mspec(R)$ 
from~\cite[Theorem~4.1]{sather:afbha}
and the seemingly more natural condition $\supp_R(F)\supseteq\VE(\fa)\bigcap\mspec(R)$.
Here $\vf^*$ is the induced map $\spec(S)\to\spec(R)$.
These ideas will be explored further in~\cite{sather:frfdgm}.

\begin{prop}\label{prop151120a}
Let $Y\in\catd(S)$.
\begin{enumerate}[\rm(a)]
\item\label{prop151120a1}
One has $\vf^*(\spec(S))=\supp_R(S)$.
\item\label{prop151120a2}
One has $\supp_R(Y)\subseteq
\vf^*(\spec(S))$.
\item\label{prop151120a3}
If $\supp_R(Y)\supseteq\VE(\fa)\bigcap\mspec(R)$, then
$\vf^*(\supp_S(Y))\supseteq\VE(\fa)\bigcap\mspec(R)$.
\end{enumerate}
\end{prop}

\begin{proof}
Let $\p\in\spec(R)$, and set $S_\p:=\Otimes{R_\p}{S}$.

\eqref{prop151120a1}
For one implication, assume that $\p\in\vf^*(\spec(S))$, and let $P\in\spec(S)$ be such that $\p=\vf^{-1}(P)$.
It follows that $P$ represents a prime ideal in the ring $S_\p/\p S_\p\cong\Otimes{\kappa(\p)}{S}$.
In particular, we have $0\neq \Otimes{\kappa(\p)}{S}\cong\HH_0(\Lotimes{\kappa(\p)}{S})$.
We conclude that $\Lotimes{\kappa(\p)}{S}\not\simeq 0$, so $\p\in\supp_R(S)$.

For the converse, assume that $\p\in\supp_R(S)$.
By definition, this implies that $\Lotimes{\kappa(\p)}{S}\not\simeq 0$, so 
there is an integer $i$ such that $\HH_i(\Lotimes{\kappa(\p)}{S})\neq 0$. 
From the isomorphism $\Lotimes{\kappa(\p)}{S}\simeq\Lotimes{(R/\p)}{S_{\p}}$,
we conclude that $\HH_i(\Lotimes{\kappa(\p)}{S})$ is a non-zero $S_{\p}$-module that is annihilated by $\p$,
that is, it is a non-zero module over $\Otimes{(R/\p)}{S_{\p}}\cong\Otimes{\kappa(\p)}{S}$.
It follows that the ring $\Otimes{\kappa(\p)}{S}$ is non-zero, so it has a prime ideal $P$,
which necessarily satisfies $\p=\vf^{-1}(P)\in\vf^*(\spec(S))$, as desired.

\eqref{prop151120a2}
Assume that $\p\in\supp_R(Y)$.
It follows that we have 
$$0\not\simeq\Lotimes{\kappa(\p)}Y\simeq\Lotimes[S]{(\Lotimes{\kappa(\p)}{S})}Y$$
so we conclude that $\Lotimes{\kappa(\p)}{S}\not\simeq 0$.
That is, by part~\eqref{prop151120a1}, we have $\p\in\supp_R(S)=\vf^*(\spec(S))$, as desired.

\eqref{prop151120a3}
Assume that $\supp_R(Y)\supseteq\VE(\fa)\bigcap\mspec(R)$, and let $\m\in\VE(\fa)\bigcap\mspec(R)$.
Let $\y$ be a finite generating sequence for $\m$, and set $L:=K^R(\y)$ and $L':=\Lotimes{S}{L}\cong K^S(\y)$.
By assumption, we have $\m\in\supp_R(Y)\bigcap\VE(\m)=\supp_R(\Lotimes YL)$, so we conclude that
$0\not\simeq\Lotimes YL\simeq\Lotimes[S]Y{L'}$.
Thus, there is a prime  $P\in\supp_S(\Lotimes[S]Y{L'})=\supp_S(Y)\bigcap\VE(\fm S)$,
so
$\vf^{-1}(P)\supseteq\vf^{-1}(\m S)\supseteq\m$. As $\m$ is maximal and $\vf^{-1}(P)$ is prime, we conclude that
$\m=\vf^{-1}(P)\in\vf^*(\supp_S(Y))$, as desired.
\end{proof}

\subsection*{Base Change and Co-base Change}
We now switch to extension of scalars.
For perspective in the results of this subsection, 
recall the characterization of $\supp_R(S)$ from Proposition~\ref{prop151120a}\eqref{prop151120a1}.

\begin{lem}\label{lem151009b}
Let $X\in\catd(R)$ be given. 
If $\supp_R(X)\subseteq\VE(\fa)$, then $\supp_S(\Lotimes SX)\subseteq\VE(\fa S)$;
the converse holds when $\supp_R(S)\supseteq\supp_R(X)$, e.g., when the map $\vf$ is faithfully flat or when it is injective and integral.
\end{lem}

\begin{proof}
As in the second paragraph of the 
proof of Lemma~\ref{lem151008a}, one has the containment 
$\supp_S(\Lotimes SX)\subseteq\VE(\fa S)$ if and only if each  module 
$\HH_i(\Lotimes SX)$ is $\fa S$-torsion, 
i.e., if and only if each homology module $\HH_i(\Lotimes SX)$ is $\fa$-torsion,
that is, if and only if $\supp_R(\Lotimes SX)\subseteq\VE(\fa)$.
Since we have $\supp_R(\Lotimes SX)=\supp_R(S)\bigcap\supp_R(X)$ by Fact~\ref{cor130528a}, the desired implications follow readily.
(When $\vf$ is faithfully flat or when it is injective and integral, we have $\supp_R(S)=\spec(R)$.)
\end{proof}

\begin{lem}\label{lem151009c}
Let $X\in\catd(R)$ be given. 
If $\VE(\fa)\supseteq\cosupp_R(X)$, then $\VE(\fa S)\supseteq\cosupp_S(\Rhom SX)$;
the converse holds when $\supp_R(S)\supseteq\cosupp_R(X)$, e.g., when the map $\vf$ is faithfully flat or when it is injective and integral.
\end{lem}

\begin{proof}
Given a prime ideal $P\in\spec(S)$ with contraction $\p\in\spec(R)$, we have the following commutative diagram of natural/induced ring homomorphisms:
$$\xymatrix{
R\ar[r]\ar[d]
&S\ar[d] \\
\kappa(\p)\ar[r]
&\kappa(P).}$$
From this, we conclude that there are isomorphisms
\begin{align*}
\Rhom[S]{\kappa(P)}{\Rhom SX}
&\simeq\Rhom{\kappa(P)}{X}\\
&\simeq\Rhom[\kappa(\p)]{\kappa(P)}{\Rhom{\kappa(\p)}{X}}.
\end{align*}
Since $\kappa(P)$ is a non-zero $\kappa(\p)$-vector space, we conclude that
$\p\in\cosupp_R(X)$ if and only if $P\in\cosupp_S(\Rhom SX)$.

Now, for one implication in the result, assume that $\cosupp_R(X)\subseteq\VE(\fa)$ and let $P\in\cosupp_S(\Rhom SX)$.
With $\p$ as above, we then have $\p\in\cosupp_R(X)\subseteq\VE(\fa)$, so we conclude that $\fa\subseteq\p$, which implies that
$\fa S\subseteq\p S\subseteq P$.
Thus, we have $P\in\VE(\fa S)$, as desired. 

Next,  assume that we have $\cosupp_S(\Rhom SX)\subseteq\VE(\fa)$ and $\supp_R(S)\supseteq\cosupp_R(X)$, and let $\p\in\cosupp_R(X)$ be given.
It follows that $\p\in\supp_R(S)=\vf^*(\spec(S))$, 
by Proposition~\ref{prop151120a}\eqref{prop151120a1},
i.e., $S$ has a prime ideal $P$ lying over $\p$.
We conclude as above that $P\in\VE(\fa S)$, so $\fa\subseteq \vf^{-1}(\fa S)\subseteq \vf^{-1}(P)=\p$, as desired.
\end{proof}

\begin{disc}\label{disc151009a}
The proof of Lemma~\ref{lem151009c} can be used to give another proof of Lemma~\ref{lem151009b}, but not vice versa.
This is due in part to the differences between~\cite[Corollary~4.32]{yekutieli:hct} and~\cite[Theorem~3]{yekutieli:sccmc}.
\end{disc}

The next result explains base-change behavior for adic finiteness.
 
\begin{thm}\label{lem151009d}
Let $X\in\catd(R)$ be given. 
If $X$ is $\fa$-adically finite over $R$ and $\Lotimes SX\in\catdb(S)$, then $\Lotimes SX$ is $\fa S$-adically finite over $S$;
the converse holds when  the map $\vf$ is flat with $\supp_R(S)\supseteq\VE(\fa)\bigcup\supp_R(X)$, e.g., when $\vf$ is faithfully flat.
\end{thm}

\begin{proof}
Note that $K':=\Lotimes SK$ is the Koszul complex over $S$ on a finite generating sequence for $\fa S$.
Consider the following isomorphisms in $\catd(S)$:
$$
\Lotimes[S]{K'}{(\Lotimes SX)}
\simeq\Lotimes{K'}{X}
\simeq\Lotimes{(\Lotimes SK)}{X}
\simeq\Lotimes S{(\Lotimes KX)}.
$$

For one implication, assume that $X$ is $\fa$-adically finite over $R$ and $\Lotimes SX\in\catdb(S)$.
Lemma~\ref{lem151009b} implies that $\supp_S(\Lotimes SX)\subseteq\VE(\fa S)$.
Our finiteness assumption implies that $\Lotimes KX\in\catdfb(R)$; using a degree-wise finite semi-free resolution of $\Lotimes KX$ over $R$,
we deduce that $\Lotimes[S]{K'}{(\Lotimes SX)}\simeq\Lotimes S{(\Lotimes KX)}\in\catdf_+(S)$.
(We actually have $\Lotimes[S]{K'}{(\Lotimes SX)}\in\catdfb(S)$, but we don't need that here.)
Thus, the $S$-complex $\Lotimes SX$ is $\fa S$-adically finite. 

Conversely, assume that $\vf$ is flat with $\supp_R(S)\supseteq\VE(\fa)$.
Assume further that the $S$-complex $\Lotimes SX$ is $\fa S$-adically finite. 
Lemma~\ref{lem151009b} implies that $\supp_R(X)\subseteq\VE(\fa)$.

Claim 1. The induced map $\ol\vf$ is faithfully flat. 
Since $\vf$ is flat, so is $\ol\vf$.
To show that it is faithfully flat, let $\m/\fa\in\mspec(R/\fa)$. 
We need to show that $\Otimes[R/\fa]{(R/\fa)/(\m/\fa)}{S/\fa}\neq 0$,
i.e., that $\Otimes{(R/\m)}S\neq 0$.
By assumption, we have $\m\in\VE(\fa)\subseteq\supp_R(S)$, so flatness gives us
$\Otimes{(R/\m)}S\simeq\Lotimes{\kappa(\m)}{S}\not\simeq 0$,
as desired.

Claim 2.
Each module $\HH_i(\Lotimes KX)$ is finitely generated over $R$.
By assumption, we have $\Lotimes[S]{K'}{(\Lotimes SX)}\in\catdfb(S)$.
By flatness, we have the following
$$\HH_i(\Lotimes[S]{K'}{(\Lotimes SX)})\cong\HH_i(\Lotimes S{(\Lotimes KX)})\cong\Otimes S{\HH_i(\Lotimes KX)}$$
so $\Otimes S{\HH_i(\Lotimes KX)}$ is finitely generated over $S$.
It follows that the next module is finitely generated over $S/\fa S$.
$$
\Otimes[S]{(S/\fa S)}{(\Otimes S{\HH_i(\Lotimes KX)})}
\cong\Otimes[R/\fa]{(S/\fa S)}{(\Otimes{(R/\fa)}{\HH_i(\Lotimes KX)})}
$$
Faithful flatness implies that $\Otimes{(R/\fa)}{\HH_i(\Lotimes KX)}\cong\HH_i(\Lotimes KX)$ is finitely generated over $R/\fa$, hence over $R$.

Now we complete the proof. By assumption, we have 
$$\Lotimes S{(\Lotimes KX)}\simeq\Lotimes[S]{K'}{(\Lotimes SX)}\in\catdfb(S).$$
Claim 2 implies that $\Lotimes KX\in\catdf(R)$.
Also, as in Claim 2, one verifies that if $\HH_i(\Lotimes[S]{K'}{(\Lotimes SX)})=0$, then $\HH_i(\Lotimes KX)=0$.
Thus, we have $\Lotimes KX\in\catdb(R)$.
From Lemma~\ref{lem150604a1}\eqref{lem150604a1a},
we have $X\in\catdb(R)$, so we conclude that $X$ is $\fa$-adically finite, as desired.
\end{proof}

As an application of the previous result, we characterize complexes with artinian total homology.
Recall that an ideal $\fb$ has \emph{finite colength} if the ring $R/\fa$ has finite length, i.e., is artinian.

\begin{prop}\label{prop160208a}
Let $X\in\catdb(R)$ be given. Then each module $\HH_i(X)$ is artinian over $R$ if and only if there is an ideal $\fb$ of $R$ with finite colength such that $X$ is $\fb$-adically finite.
\end{prop}

\begin{proof}
Assume first that each $R$-module $\HH_i(X)$ is artinian.
Since we have $X\in\catdb(R)$, a result of Sharp~\cite[Proposition~1.4]{sharp:msamaab} implies that there is a finite list $\m_1,\ldots,\m_n$ of maximal ideals of $R$ such that
$\HH_i(X)\cong\oplus_{j=1}^n\Gamma_{\m_j}(\HH_i(X))$ for all $i$. 
Set $\fb=\bigcap_{j=1}^n\m_j$, which has finite colength. 
The above isomorphism implies that each module $\HH_i(X)$ is $\fb$-torsion, so we have $\supp_R(X)\subseteq\VE(\fb)$ by Fact~\ref{cor130528a}. 
To show that $X$ is $\fb$-adically finite, it remains to show that $\Lotimes KX\in\catdf(R)$. 
Using a standard cone/truncation argument, it suffices to show this in the case where $X\simeq\HH_0(X)$ is a module. 
Using the above direct sum decomposition, we reduce to the case where the module $X$ is $\m_j$-torsion.
In this case, the module $X_{\m_j}$ has a natural $R_{\m_j}$-module structure and is artinian over both $R$ and $R_{\m_j}$.
By construction, we have $\b R_{\m_j}=\m_j R_{\m_j}$ and $K_{\m_j}\simeq\Lotimes{R_{\m_j}}K$ is the Koszul complex over $R_{\m_j}$ on a finite generating sequence for $\m_j R_{\m_j}$.
Thus, from the local case in Example~\ref{ex160206a}\eqref{ex160206a3}, we conclude that $\Lotimes KX\simeq\Lotimes[R_{\m_j}]{(\Lotimes{R_{\m_j}}K)}{\!\! \!\!X}\in\catdf(R_{\m_j})$.
It follows that $\Lotimes KX$ has homology of finite length over $R_{\m_j}$ and over $R$, so we have $\Lotimes KX\in\catdf(R)$, as desired.

Conversely, assume that $R$ has an ideal $\fb$ of finite colength such that $X$ is $\fb$-adically finite. 
Using~\cite[Proposition~7.1]{sather:scc}, we may replace $\fb$ with $\rad{\fb}$ to assume that $\fb$ is an intersection of finitely many maximal ideals of $R$, say $\fb=\bigcap_{j=1}^n\m_j$.
Theorem~\ref{lem151009d} implies that $\Lotimes{R_{\m_j}}X$ is $\b R_{\m_j}$-adically finite (i.e., $\m_j R_{\m_j}$-adically finite) over $R_{\m_j}$ for each $j$.
Example~\ref{ex160206a}\eqref{ex160206a3} implies that each module $\HH_i(\Lotimes{R_{\m_j}}X)\cong\HH_i(X)_{\m_j}$ is artinian over $R_{\m_j}$. 
Since each module $\HH_i(X)$ is $\fb$-torsion, we have $\Supp_R(\HH_i(X))\subseteq\VE(\fb)=\{\m_1,\ldots,\m_n\}$.
So, we conclude that each $\HH_i(X)$ is artinian over $R$ by~\cite[Lemma~3.2]{kubik:hamm2}.
\end{proof}

\section{Adic Finiteness and Homological Dimensions}\label{sec151206a}

The point of this section is to show that certain homological dimension computations extend from the setting of
homologically finite complexes to the realm of $\fa$-adically finite complexes. 
To begin, we prove Theorem~\ref{prop151115aa} from the introduction.
Note that the special case $Z=\RG aR$ of this result is also well-known, using the ``telescope complex''.
However, the general result showcases the power (or, if one prefers, the restrictiveness) of adic finiteness in general. 

\begin{thm}\label{prop151115a}
Let $X\in\catdb(R)$ be $\fa$-adically finite. The following  are equivalent.
\begin{enumerate}[\rm(i)]
\item\label{prop151115a1}
For all $\m\in\mspec(R)\bigcap\VE(\fa)$, one has $\fd_{R_\m}(X_\m)<\infty$.
\item\label{prop151115a2}
For all $\m\in\mspec(R)\bigcap\VE(\fa)$, one has $\pd_{R_\m}(X_\m)<\infty$.
\item\label{prop151115a3}
One has $\fd_R(X)<\infty$.
\item\label{prop151115a4}
One has $\pd_R(X)<\infty$.
\end{enumerate}
Moreover, one has $\pd_R(X)=\fd_R(X)$.
\end{thm}

\begin{proof}
\eqref{prop151115a3}$\iff$\eqref{prop151115a4}
By Lemma~\ref{lem151206a}, we have
the first and last equalities next.
\begin{align*}
n+\fd_R(X)
&=\fd_R(\Lotimes KX)
=\pd_R(\Lotimes KX)
=n+\pd_R(X)
\end{align*}
The second equality is from the assumption $\Lotimes KX\in\catdfb(R)$.

From  Theorem~\ref{lem151009d} we know that for all $\m\in\mspec(R)\bigcap\VE(\fa)$, the $R_\m$-complex $X_\m\simeq\Lotimes{R_\m}X$ is $\fa R_\m$-adically finite. 
Thus, the equivalence of conditions~\eqref{prop151115a1} and~\eqref{prop151115a2} follows from what we have just shown.
The implication~\eqref{prop151115a4}$\implies$\eqref{prop151115a2} is from~\cite[Proposition~5.1(P)]{avramov:hdouc}, so it remains to prove
the implication~\eqref{prop151115a2}$\implies$\eqref{prop151115a4}.

Assume that for all $\m\in\mspec(R)\bigcap\VE(\fa)$, one has $\pd_{R_\m}(X_\m)<\infty$.
It follows that for all $\m\in\mspec(R)\bigcap\VE(\fa)$, we have
$$\pd_{R_\m}((\Lotimes KX)_\m)=\pd_{R_\m}(\Lotimes[R_\m]{K_\m}{X_\m})<\infty.$$
From the condition $\Lotimes KX\in\catdfb(R)$, we have the finiteness in the next display
$$\pd_R(X)=\pd_R(\Lotimes KX)-n<\infty$$
and the equality is from Lemma~\ref{lem151206a}.
\end{proof}

\begin{disc}\label{disc151207a}
It is worth noting that in conditions~\eqref{prop151115a1} and~\eqref{prop151115a2} of Proposition~\ref{prop151115b}
one can replace $\mspec(R)\bigcap\VE(\fa)$ with other sets, e.g., $\mspec(R)\bigcap\supp_R(X)$ or $\supp_R(X)$ or $\VE(\fa)$ or $\mspec(R)$. 
Indeed, the fact that $X$ is $\fa$-adically finite implies that $\supp_R(X)=\Supp_R(X)$, by~\cite[Theorem~7.11]{sather:scc}.
Since $X_{\p}\simeq 0$ if and only if $\p\notin\Supp_R(X)=\supp_R(X)\subseteq\VE(\fa)$, it is straightforward to show that $\mspec(R)\bigcap\VE(\fa)$ can be replaced by 
any of the sets listed.

Similarly, in the next result,
one can replace $\spec(R)$ with $\supp_R(X)$ or $\VE(\fa)$; in the expression $\sup\{\pd_{R_\p}(X_\p)\mid\p\in\spec(R)\}$, one can even replace $\spec(R)$
with $\mspec(R)$ or $\mspec(R)\bigcap\VE(\fa)$ or $\mspec(R)\bigcap\supp_R(X)$.
This result extends~\cite[Proposition~5.3.P]{avramov:hdouc} to our setting.
\end{disc}

\begin{prop}\label{prop151115b}
Let $X\in\catdb(R)$ be $\fa$-adically finite. 
Then we have
\begin{align*}
\pd_R(X)
&=\sup\{-\inf(\Rhom{X}{R/\p})\mid\p\in\spec(R)\}\\
&=\sup\{-\inf(\Rhom[R_\p]{X_{\p}}{\kappa(\p)})\mid\p\in\spec(R)\}\\
&=\sup\{\pd_{R_\p}(X_\p)\mid\p\in\spec(R)\}.
\end{align*}
\end{prop}

\begin{proof}
Let $\p\in\spec(R)$ be given. As in the proof of Lemma~\ref{lem151206a}, we have
$$\cosupp_R(\Rhom X{R/\p})\subseteq\supp_R(X)\subseteq\VE(\fa)$$
so we conclude that
\begin{align*}
\inf(\Rhom{\Lotimes KX}{R/\p})
&=\inf(\Rhom K{\Rhom X{R/\p}})\\
&=\inf(\shift^{-n}\Lotimes K{\Rhom X{R/\p}})\\
&=-n+\inf(\Lotimes K{\Rhom X{R/\p}})\\
&=-n+\inf(\Rhom X{R/\p}).
\end{align*}
This explains the third step in the next sequence
\begin{align*}
n+\pd_R(X)
&=\pd_R(\Lotimes KX)\\
&=\sup\{-\inf(\Rhom{\Lotimes KX}{R/\p})\mid\p\in\spec(R)\}\\
&=n+\sup\{-\inf(\Rhom X{R/\p})\mid\p\in\spec(R)\}\\
&\leq n+\sup\{-\inf(\Rhom X{N})\mid\text{$N$ is an $R$-module}\}\\
&=n+\pd_R(X)
\end{align*}
The first step is from Lemma~\ref{lem151206a}, and the second step is from~\cite[Proposition~5.3.P]{avramov:hdouc},
using the assumption $\Lotimes KX\in\catdfb(R)$.
The fourth step is routine, and the last one is from~\cite[Corollary~2.5.P]{avramov:hdouc}.
This explains the first equality in our result. 

For the other equalities in the statement of our result,
we use the condition $\supp_{R_\p}(X_\p)\subseteq\VE(\fa R_\p)$ from Lemma~\ref{lem151009b}
to compute similarly
\begin{align*}
-\inf(\Rhom[R_\p]{(\Lotimes KX)_{\p}}{\kappa(\p)})
&= -\inf(\Rhom[R_\p]{\Lotimes[R_{\p}] {K_{\p}}{X_{\p}}}{\kappa(\p)}) \\
&= -\inf(\Lotimes[R_{\p}] {\shift^{-n}K_{\p}}{\Rhom[R_\p]{X_{\p}}{\kappa(\p)}}) \\
&=n-\inf(\Rhom[R_\p]{X_{\p}}{\kappa(\p)})
\end{align*}
and from this we have
\begin{align*}
n+\pd_R(X)\hspace{-10mm} \\
&=\pd_R(\Lotimes KX)\\
&=\sup\{-\inf(\Rhom[R_\p]{(\Lotimes KX)_{\p}}{\kappa(\p)})\mid\p\in\spec(R)\}\\
&=n+\sup\{-\inf(\Rhom[R_\p]{X_{\p}}{\kappa(\p)})\mid\p\in\spec(R)\}\\
&\leq n+\sup\{-\inf(\Rhom[R_\p]{X_{\p}}{N})\mid\text{$\p\in\spec(R)$ and $N$ is an $R_\p$-module}\}\\
&=n+\sup\{\pd_{R_{\p}}(X_{\p})\mid\p\in\spec(R)\} \\
&\leq n+\pd_R(X)
\end{align*}
and hence the desired result.
\end{proof}

In the next result, we tackle~\cite[Proposition~5.5]{avramov:hdouc}.
Note that the local assumption on $\vf$ implies that $\fa S\neq S$.

\begin{prop}\label{prop151207a}
Let $\vf\colon(R,\m,k)\to S$ be a local homomorphism, and let $Y\in\catdb(S)$ be $\fa S$-adically finite.
\begin{enumerate}[\rm(a)]
\item \label{prop151207a1}
There are equalities
\begin{align*}
\fd_R(Y)&=\sup(\Lotimes kY) &
\id_R(Y)&=-\inf(\Rhom kY).
\end{align*}
\item \label{prop151207a2}
If $\vf$ is the identity map, then
$$\pd_R(Y)=-\inf(\Rhom Yk).$$
\end{enumerate}
\end{prop}

\begin{proof}
Argue as in the first paragraph of the proof of Theorem~\ref{prop151115a}, using 
Lemmas~\ref{lem151206a}--\ref{lem151206b} to reduce to the homologically finite case of~\cite[Proposition~5.5]{avramov:hdouc}.
\end{proof}

\section*{Acknowledgments}
We are grateful to Srikanth Iyengar, 
Liran Shaul,
and Amnon Yekutieli
for helpful comments about this work.

%\bibliography{../+new}

\begin{thebibliography}{10}

\bibitem{lipman:lhcs}
L.\ {Alonso Tarr{\'{\i}}o}, A.~Jerem{\'{\i}}as L{\'o}pez, and J.\ Lipman,
  \emph{Local homology and cohomology on schemes}, Ann. Sci. \'Ecole Norm. Sup.
  (4) \textbf{30} (1997), no.~1, 1--39. \MR{1422312 (98d:14028)}

\bibitem{avramov:hdouc}
L.~L. Avramov and H.-B.\ Foxby, \emph{Homological dimensions of unbounded
  complexes}, J. Pure Appl. Algebra \textbf{71} (1991), 129--155.
  \MR{93g:18017}

\bibitem{avramov:dgha}
L.~L. Avramov, H.-B.\ Foxby, and S.\ Halperin, \emph{Differential graded
  homological algebra}, in preparation.

\bibitem{benson:lcstc}
D.~Benson, S.~B. Iyengar, and H.~Krause, \emph{Local cohomology and support for
  triangulated categories}, Ann. Sci. \'Ec. Norm. Sup\'er. (4) \textbf{41}
  (2008), no.~4, 573--619. \MR{2489634 (2009k:18012)}

\bibitem{benson:csc}
\bysame, \emph{Colocalizing subcategories and cosupport}, J. Reine Angew. Math.
  \textbf{673} (2012), 161--207. \MR{2999131}

\bibitem{christensen:apac}
L.~W. Christensen and H.\ Holm, \emph{Ascent properties of {A}uslander
  categories}, Canad. J. Math. \textbf{61} (2009), no.~1, 76--108. \MR{2488450}

\bibitem{foxby:ibcahtm}
H.-B.\ Foxby, \emph{Isomorphisms between complexes with applications to the
  homological theory of modules}, Math. Scand. \textbf{40} (1977), no.~1,
  5--19. \MR{0447269 (56 \#5584)}

\bibitem{foxby:bcfm}
\bysame, \emph{Bounded complexes of flat modules}, J. Pure Appl. Algebra
  \textbf{15} (1979), no.~2, 149--172. \MR{535182 (83c:13008)}

\bibitem{foxby:daafuc}
H.-B.\ Foxby and S.\ Iyengar, \emph{Depth and amplitude for unbounded
  complexes}, Commutative algebra. Interactions with Algebraic Geometry,
  Contemp. Math., vol. 331, Amer. Math. Soc., Providence, RI, 2003,
  pp.~119--137. \MR{2 013 162}

\bibitem{hartshorne:rad}
R.~Hartshorne, \emph{Residues and duality}, Lecture Notes in Mathematics, No.
  20, Springer-Verlag, Berlin, 1966. \MR{36 \#5145}

\bibitem{hartshorne:lc}
\bysame, \emph{Local cohomology}, A seminar given by A. Grothendieck, Harvard
  University, Fall, vol. 1961, Springer-Verlag, Berlin, 1967. \MR{0224620 (37
  \#219)}

\bibitem{hartshorne:adc}
\bysame, \emph{Affine duality and cofiniteness}, Invent. Math. \textbf{9}
  (1969/1970), 145--164. \MR{0257096 (41 \#1750)}

\bibitem{kawasaki:ccma}
K.-i. Kawasaki, \emph{On a category of cofinite modules which is {A}belian},
  Math. Z. \textbf{269} (2011), no.~1-2, 587--608. \MR{2836085 (2012h:13026)}

\bibitem{kawasaki:ccc}
\bysame, \emph{On a characterization of cofinite complexes. {A}ddendum to
  ``{O}n a category of cofinite modules which is {A}belian''}, Math. Z.
  \textbf{275} (2013), no.~1-2, 641--646. \MR{3101824}

\bibitem{kubik:hamm2}
B.~Kubik, M.~J. Leamer, and S.~Sather-Wagstaff, \emph{Homology of {A}rtinian
  and mini-max modules, {II}}, J. Algebra \textbf{403} (2014), 229--272.
  \MR{3166074}

\bibitem{lipman:llcd}
J.\ Lipman, \emph{Lectures on local cohomology and duality}, Local cohomology
  and its applications (Guanajuato, 1999), Lecture Notes in Pure and Appl.
  Math., vol. 226, Dekker, New York, 2002, pp.~39--89. \MR{1888195
  (2003b:13027)}

\bibitem{matlis:kcd}
E.~Matlis, \emph{The {K}oszul complex and duality}, Comm. Algebra \textbf{1}
  (1974), 87--144. \MR{0344241 (49 \#8980)}

\bibitem{matlis:hps}
\bysame, \emph{The higher properties of {$R$}-sequences}, J. Algebra
  \textbf{50} (1978), no.~1, 77--112. \MR{479882 (80a:13013)}

\bibitem{melkersson:mci}
Leif Melkersson, \emph{Modules cofinite with respect to an ideal}, J. Algebra
  \textbf{285} (2005), no.~2, 649--668. \MR{2125457 (2006i:13033)}

\bibitem{yekutieli:hct}
M.~Porta, L.~Shaul, and A.~Yekutieli, \emph{On the homology of completion and
  torsion}, Algebr. Represent. Theory \textbf{17} (2014), no.~1, 31--67.
  \MR{3160712}

\bibitem{sather:frfdgm}
S.~Sather-Wagstaff, \emph{Fidelity results for {DG} modules}, in preparation.

\bibitem{sather:afbha}
S.~Sather-Wagstaff and R.~Wicklein, \emph{Adic finiteness: Bounding homology
  and applications}, preprint (2016), \texttt{arxiv:1602.03225}.

\bibitem{sather:afc}
\bysame, \emph{Adic {F}oxby classes}, preprint (2016),
  \texttt{arxiv:1602.03227}.

\bibitem{sather:asc}
\bysame, \emph{Adic semidualizing complexes}, preprint (2015),
  \texttt{arxiv:1506.07052}.

\bibitem{sather:elclh}
\bysame, \emph{Extended local cohomology and local homology}, preprint (2016),
  \texttt{arxiv:1602.03226}.

\bibitem{sather:scc}
\bysame, \emph{Support and adic finiteness for complexes}, Comm. Algebra, to
  appear, \texttt{arXiv:1401.6925}.

\bibitem{sharp:msamaab}
R.~Y. Sharp, \emph{A method for the study of {A}rtinian modules, with an
  application to asymptotic behavior}, Commutative algebra ({B}erkeley, {CA},
  1987), Math. Sci. Res. Inst. Publ., vol.~15, Springer, New York, 1989,
  pp.~443--465. \MR{1015534 (91a:13011)}

\bibitem{verdier:cd}
J.-L.\ Verdier, \emph{Cat\'{e}gories d\'{e}riv\'{e}es}, SGA 4$\frac{1}{2}$,
  Springer-Verlag, Berlin, 1977, Lecture Notes in Mathematics, Vol. 569,
  pp.~262--311. \MR{57 \#3132}

\bibitem{verdier:1}
\bysame, \emph{Des cat\'egories d\'eriv\'ees des cat\'egories ab\'eliennes},
  Ast\'erisque (1996), no.~239, xii+253 pp. (1997), With a preface by Luc
  Illusie, Edited and with a note by Georges Maltsiniotis. \MR{98c:18007}

\bibitem{yekutieli:sccmc}
A.~Yekutieli, \emph{A separated cohomologically complete module is complete},
  Comm. Algebra \textbf{43} (2015), no.~2, 616--622. \MR{3274025}

\end{thebibliography}
\providecommand{\bysame}{\leavevmode\hbox to3em{\hrulefill}\thinspace}
\providecommand{\MR}{\relax\ifhmode\unskip\space\fi MR }
\providecommand{\MRhref}[2]{%
  \href{http://www.ams.org/mathscinet-getitem?mr=#1}{#2}
}
\providecommand{\href}[2]{#2}

\end{document}